\documentclass[leqno,english]{amsart}
\usepackage{amsfonts,amssymb,amsmath,amsgen,amsthm}

\usepackage{tikz,tkz-tab}
\usepackage{hyperref}
\usepackage{color}
\newcommand{\msc}[2][2000]{%
  \let\@oldtitle\@title%
  \gdef\@title{\@oldtitle\footnotetext{#1 \emph{Mathematics subject
        classification.} #2}}%
}

\theoremstyle{plain}
\newtheorem{theorem}{Theorem}[section]

\newtheorem{lemma}[theorem]{Lemma}
\newtheorem{corollary}[theorem]{Corollary}

\theoremstyle{remark}
\newtheorem{remark}[theorem]{Remark}

\def\C{{\mathbb C}}
\def\R{{\mathbb R}}
\def\N{{\mathbb N}}
\def\H{{\mathcal H}}
\def\O{\mathcal O}
\def\F{\mathcal F}

\def\({\left(}
\def\){\right)}
\def\<{\left\langle}
\def\>{\right\rangle}
\def\le{\leqslant}
\def\ge{\geqslant}

\def\Tend#1#2{\mathop{\longrightarrow}\limits_{#1\rightarrow#2}}

\def\d{{\partial}}
\def\eps{\varepsilon}

\def\si{{\sigma}}

\def\1{{\mathbf 1}}
\def\w{{w}}

\DeclareMathOperator{\RE}{Re}
\DeclareMathOperator{\IM}{Im}

\def\FF{F}
\def\GG{G}

\numberwithin{equation}{section}

\newcommand{\mynegspace}{\hspace{-0.12em}}
\newcommand{\vvvert}{\rvert\mynegspace\rvert\mynegspace\rvert}

\begin{document}

\title[WKB analysis of generalized  DNLS]{WKB analysis of
  generalized derivative nonlinear Schr\"odinger equations without
  hyperbolicity}   

\author[R. Carles]{R\'emi Carles}
\author[C. Gallo]{Cl\'ement Gallo}
\address{CNRS \& Univ. Montpellier\\Institut Montpelli\'erain
  Alexander Grothendieck
\\CC51\\Place E. Bataillon\\34095 Montpellier\\ France}
\email{Remi.Carles@math.cnrs.fr}
\email{Clement.Gallo@umontpellier.fr}

\begin{abstract}
We consider the semi-classical limit of nonlinear Schr\"odinger
equations in the presence of both a polynomial nonlinearity and the
derivative in space of a polynomial nonlinearity. By working in a class
of analytic initial data, we do not have to assume any hyperbolic
structure on the (limiting) phase/amplitude system. The solution, its
approximation, and 
the error estimates are considered in time dependent analytic
regularity. 
\end{abstract}
\thanks{This work was supported  by the French ANR projects
  BECASIM
  (ANR-12-MONU-0007-04) and BoND (ANR-13-BS01-0009-01).}  
\maketitle

\section{Introduction}

\subsection{Setting}
\label{sec:setting}

We consider the equation
\begin{equation}
  \label{eq:nls}
  i\eps\d_t u^\eps +\frac{\eps^2}{2}\partial_x^2 u^\eps +i\frac{\eps}{2}\partial_x\left(g(|u^\eps|^2)u^\eps\right)-f(|u^\eps|^2)u^\eps= 0,\quad (t,x)\in [0,T]\times \R,
\end{equation}
in the semi-classical limit $\eps \to 0$, where $f,g$ are polynomials
and $g(0)=f(0)=0$. The typical example we consider here is 
\begin{equation}
  \label{eq:NL}
  g(s)=\alpha
s^\gamma,\quad f(s)=\lambda s^\sigma,\quad \text{where }\alpha,\lambda
\in\R \text{ and }
\gamma,\sigma\in\N\setminus\{0\},
\end{equation}
but $f$ and $g$ need not be monomials.
The initial data that we 
consider are WKB states (also known as Lagrangian states):
\begin{equation}
  \label{eq:ci-u}
  u^\eps(0,x)=a_0^\eps(x)e^{i\phi_0^\eps(x)/\eps}=:u_0^\eps(x),
\end{equation}
where $\phi_0^\eps:\R\to \R$ is a real-valued phase, and $a_0^\eps:\R\to \C$
is a possibly complex-valued amplitude. Our goal is to understand the
semi-classical limit of equation (\ref{eq:nls}), that is to describe
the behaviour in the limit $\eps\to 0$ of the solutions to
(\ref{eq:nls}) with initial data (\ref{eq:ci-u}). We consider
$\eps$-dependent initial phase and amplitude, but they can be thought
of as $\eps$-independent, or having an asymptotic in powers of $\eps$,
as will be discussed below.  

In the case $g=0$, we recover the more standard nonlinear
Schr\"odinger equation, modelling for instance Bose--Einstein
condensation (see e.g. \cite{JosserandPomeau,PiSt}). The case $f=0$
corresponds to the derivative Schr\"odinger equation, describing
Alfv\'en waves (see e.g. \cite{KhRa82,MOMT76,Mj76,Sulem}). The cubic
cases, 
$\alpha=0$ with $\si=1$, or 
$\gamma=\sigma=1$, with $\lambda=-1$, are known to be completely
integrable (\cite{AblowitzClarkson,WaKoIc79,ZS}). Derivation and
analysis in non-cubic cases can 
be found in e.g. \cite{DuLaSz16,HaOz94,LiSiSu13,Oz96}. 

\subsection{Formal limit: hydrodynamical structure}
\label{sec:formal}

Assuming that $a_0^\eps$ is real-valued, the standard approach known
as Madelung 
transform consists in seeking the solution $u^\eps$ under the form
$u^\eps= a^\eps e^{i\phi^\eps/\eps}$ with $a^\eps$ and $\phi^\eps$
real-valued. Plugging such an expression into \eqref{eq:nls} and
separating real and imaginary parts yields
\begin{equation}\label{eq:syst-madelung}
  \left\{
    \begin{aligned}
      &\d_t \phi^\eps + \frac{1}{2}(\partial_x
      \phi^\eps)^2+\frac{1}{2}g\(|a^\eps|^2\)\partial_x\phi^\eps +f\(
      |a^\eps|^2\)=\frac{\eps^2}{2} \frac{\d_x^2 a^\eps}{a^\eps}  ,\quad \phi^\eps_{\mid t=0}=\phi_0^\eps,\\
& \d_t a^\eps + \d_x \phi^\eps \d_x a^\eps
+\frac{1}{2}a^\eps\d_x^2 \phi^\eps+\frac{1}{2}\d_x\( g\(|a^\eps|^2\)a^\eps\)=
0,\quad a^\eps_{\mid t=0}=a_0^\eps.
    \end{aligned}
\right.
\end{equation}
Making this approach rigorous can be a delicate issue, especially when
$a^\eps$ has zeroes (see \cite{CaDaSa12} and references
therein). Remaining at a formal level, in the limit $\eps\to 
0$, the quantum pressure (right hand side of the equation for
$\phi^\eps$) vanishes, and we get
\begin{equation}\label{eq:syst-limit}
  \left\{
    \begin{aligned}
      &\d_t \phi + \frac{1}{2}(\partial_x
      \phi)^2+\frac{1}{2}g\(|a|^2\)\partial_x\phi +f\(
      |a|^2\)=0,\quad \phi_{\mid t=0}=\phi_0,\\
& \d_t a + \d_x \phi \d_x a
+\frac{1}{2}a\d_x^2 \phi+\frac{1}{2}\d_x\( g\(|a|^2\)a\)=
0,\quad a_{\mid t=0}=a_0,
    \end{aligned}
\right.
\end{equation}
where we have supposed also that the initial phase and amplitude
converge, $\phi_0^\eps\to \phi_0$ and $a_0^\eps \to a_0$, as $\eps\to
0$. Note that this formal convergence remains in the case where
$a_0^\eps$ is complex-valued. Introducing $\rho=|a|^2$ and $v=\d_x
\phi$, \eqref{eq:syst-limit} yields
\begin{equation}
  \label{eq:euler-gen}
  \left\{
    \begin{aligned}
      & \d_t v + v\d_x v +\frac{1}{2}\d_x\(g(\rho)v\) +\d_x
      f(\rho)=0,\quad v_{\mid t=0}=\d_x \phi_0,\\
& \d_t \rho+\d_x(\rho v) +\d_xQ(\rho)=0,\quad \rho_{\mid t=0}=|a_0|^2,
    \end{aligned}
\right.
\end{equation}
where 
\begin{equation*}
  Q(\rho) = \rho g(\rho) -\frac{1}{2}\int_0^\rho g(r)dr.
\end{equation*}
This is a generalized compressible Euler equation. We recover the
standard isentropic Euler equation when $g=0$, with pressure law
$p(\rho) =\rho f(\rho)-\int_0^\rho f(r)dr$. 
\subsection{Rigorous limit: mathematical setting}
\label{sec:rigor1}

As noticed in \cite{Grenier98}, if $g=0$ and $f\in C^\infty$ is such
that $f'>0$ (not necessarily assuming that $f$ is polynomial), the
system \eqref{eq:euler-gen} is hyperbolic. Based on this important
remark, it is possible to justify the semi-classical limit in Sobolev
spaces $H^s(\R)$, locally in time (so long as the solution to the
Euler equation \eqref{eq:euler-gen} remains smooth, that is, in
particular, on a time interval independent of $\eps$). The assumption
$f'>0$ was relaxed to cases where \eqref{eq:euler-gen} is hyperbolic
with $f'\ge 0$ (the nonlinearity need not be cubic at the origin) in
\cite{ACARMA,ChironRousset}. 
\smallbreak

The idea of Grenier consists in modifying the Madelung transform, by
allowing the amplitude $a^\eps$ to be complex-valued, and taking
advantage to this new degree of freedom (compared to the Madelung
transform) to consider 
\begin{equation}\label{eq:syst-grenier}
  \left\{
    \begin{aligned}
      &\d_t \phi^\eps + \frac{1}{2}(\partial_x
      \phi^\eps)^2+\frac{1}{2}g\(|a^\eps|^2\)\partial_x\phi^\eps +f\(
      |a^\eps|^2\)=0  ,\quad \phi^\eps_{\mid t=0}=\phi_0^\eps,\\ 
& \d_t a^\eps + \d_x \phi^\eps \d_x a^\eps
+\frac{1}{2}a^\eps\d_x^2 \phi^\eps+\frac{1}{2}\d_x\( g\(|a^\eps|^2\)a^\eps\)=
\frac{i\eps}{2}\d_x^2 a^\eps,\quad a^\eps_{\mid t=0}=a_0^\eps.
    \end{aligned}
\right.
\end{equation}
We have written directly the system in the presence of $g$, in view of
future references. It is readily checked that if $(\phi^\eps,a^\eps)$
solves \eqref{eq:syst-grenier}, 
then $u^\eps= a^\eps e^{i\phi^\eps/\eps}$ solves \eqref{eq:nls}. As
suggested above, the good unknown to work in Sobolev spaces is not
$(\phi^\eps,a^\eps)$, but rather $(\d_x \phi^\eps,a^\eps)$, or even
$(\d_x \phi^\eps,\RE a^\eps,\IM a^\eps)$. The system satisfied by this
unknown (readily obtained from  \eqref{eq:syst-grenier}) is a
\emph{skew-symmetric} perturbation of (the symmetric version of)
\eqref{eq:euler-gen}. 
 \smallbreak

In the case $g\not =0$, the semi-classical limit for \eqref{eq:nls} was
considered in  \cite{DeLiTs00} (case $f=0$) and \cite{DesjardinsLin}
(with $f,g\in C^\infty(\R_+;\R)$), by considering
\eqref{eq:syst-grenier}. However, in the case where $g\not =0$,
hyperbolicity is not a property that one has for free. In
\cite{DeLiTs00} (case $f=0$), the semi-classical analysis relies on
the assumption 
\begin{equation*}
  \d_x \phi^\eps g'>0,
\end{equation*}
where $\phi^\eps$ appears in \eqref{eq:syst-grenier}, and in
\cite{DesjardinsLin}, it relies on 
\begin{equation*}
  \d_x \phi^\eps g'+f'>0.
\end{equation*}
These assumptions are made to ensure the hyperbolicity for
\eqref{eq:syst-grenier}, but have the strong drawback to involve the
solution itself. 
\smallbreak

To overcome this issue, we work in a functional setting where
hyperbolicity is not needed. Assume $g=0$: if $f'<0$ ($\lambda<0$ in
\eqref{eq:NL}), then the Euler equation  \eqref{eq:euler-gen} is
elliptic. G.~M\'etivier \cite{GuyCauchy} has proved that in this case,
the only reasonable $C^1$ solutions to \eqref{eq:euler-gen} stem from
\emph{analytic} initial data. Indeed, if $\phi_0$ is analytic at some
point $x_0\in \R$ and  \eqref{eq:euler-gen} has a $C^1$ solution, then
$a_0$ is analytic at $x_0$. Therefore, if $\phi_0$ is analytic
(e.g. $\phi_0=0$) and $a_0$ is not, then  \eqref{eq:euler-gen} has no
$C^1$ solution. Conversely, if the initial data $a_0^\eps$ and
$\phi_0^\eps$ are analytic, then the semi-classical limit for
\eqref{eq:nls} with $g=0$ was studied in \cite{PGX93,ThomannAnalytic},
thanks to some tools developed by J.~Sj\"ostrand \cite{Sj82}, based
on complex analysis. We shall also work with analytic regularity, but
rather with a Fourier analysis point of view, introduced by J.~Ginibre
and G.~Velo \cite{GV01}. 
\smallbreak

Following \cite{GV01}, for $\w \ge  0$ and $\ell\ge 0$, we consider the space
\begin{equation*}
\H_\w^\ell=\{\psi\in L^2(\R),\quad \|\psi\|_{\H_\w^\ell}<\infty\},\quad
\text{where}\quad
  \|\psi\|_{\H_\w^\ell}^2 := 
               \int_{\R}\<\xi\>^{2\ell}e^{2\w\<\xi\>}   |\hat\psi(\xi)|^2d\xi,
\end{equation*}
with $\<\xi\>=\sqrt{1+\xi^2}$, and where the Fourier transform is defined by
\begin{equation*}
  \hat \psi(\xi)=\F \psi(\xi)=\frac{1}{\sqrt{2\pi}}\int_\R e^{-ix\xi}\psi(x)dx.
\end{equation*}
Note that if $\ell\ge 0$ and $\w>0$, $\H_\w^\ell$ is
continuously embedded in all Sobolev spaces $H^s$ for $s\in\R$.
The interest of considering a time-dependent, decreasing, weight $\w$
is that energy estimates become similar to parabolic estimates, since
\begin{equation}\label{eq:evol-norm}
  \frac{d}{dt}\|\psi\|_{\H_\w^\ell}^2 = 2\RE\<\psi,\d_t
  \psi\>_{\H_\w^\ell} +2\dot \w \|\psi\|_{\H_\w^{\ell+1/2}}^2.
\end{equation}
The last term may be understood as a gain of regularity ($\dot \w<0$).  
Mimicking our approach in \cite{CaGa-p} (where convergence results for
a numerical scheme in the semi-classical limit are established in the
case $g=0$), we will consider 
solutions to (\ref{eq:syst-grenier}) where the phase 
and the complex-valued amplitude both live in such spaces, for a weight
$\w=\w(t)=\w_0-Mt$, where $\w_0>0$ and $M>0$ are fixed. Such
spaces are also reminiscent of the framework considered in
\cite{MoVi11}. More precisely, for $T>0$, we will work in spaces such as
\begin{equation*}
C([0,T],\H_\w^\ell)=\left\{\psi\ |\
  \mathcal{F}^{-1}\(e^{\w(t)\<\xi\>}\hat{\psi}\)\in
  C([0,T],\H_0^\ell)=C([0,T],H^\ell)\right\}, 
\end{equation*}
where $H^\ell=H^\ell(\R)$ is the standard Sobolev space, or
\begin{equation*}
L^2([0,T],\H_\w^\ell)=L^2_T\H_\w^\ell=\left\{\psi\ |\
  \int_0^T\|\psi(t)\|_{\H_{\w(t)}^\ell}^2dt<\infty\right\}.
\end{equation*}
Phases and amplitudes belong to spaces
\begin{align*}
Y_{\w,T}^\ell=C([0,T],\H_\w^{\ell})\cap
L^2_T \H_\w^{\ell+1/2},
\end{align*}
and the fact that phase and amplitude do not have exactly the same
regularity shows up in the introduction of the space
\begin{align*}
X_{\w,T}^\ell=Y_{\w,T}^{\ell+1}\times Y_{\w,T}^\ell,
\end{align*}
which is reminiscent of the fact that in the hyperbolic case, the good
unknown is $(\d_x \phi^\eps,a^\eps)$ rather than $(\phi^\eps,a^\eps)$. 
The space $X_{\w,T}^\ell$ is endowed with the norm 
$$\|(\phi,a)\|_{X_{\w,T}^\ell}=\vvvert\phi\vvvert_{\ell+1,T}+\vvvert
a\vvvert_{\ell,T}, $$ 
where
\begin{equation}\label{eq:triple}
  \vvvert \psi\vvvert_{\ell,t}^2 = \max\left(\sup_{0\le s\le
    t}\|\psi(s)\|_{\H_{\w(s)}^\ell} ^2,2M\int_0^t 
  \|\psi(s)\|_{\H_{\w(s)}^{\ell+1/2}}^2 ds\right). 
\end{equation}

\subsection{Main results}
\label{sec:main-results}

Our first result states local well-posedness for
(\ref{eq:syst-grenier}) in this functional framework. To lighten our
statements as well as the proofs, we shall assume that $f$ and $g$ are
of the form \eqref{eq:NL}, but linear combinations of such functions
could be addressed as well, with heavier notations only.
\begin{theorem}\label{th:wp}
Let $\w_0>0$, $\ell>1$ and $(\phi_0^\eps,a_0^\eps)_{\eps\in[0,1]}$ be
a bounded family in $\H_{\w_0}^{\ell+1}\times\H_{\w_0}^{\ell}$. Then,
provided $M=M(\ell)>0$ is chosen sufficiently large, for all $\eps\in [0,1]$,
there is a unique solution $(\phi^\eps,a^\eps)\in X_{\w,T}^\ell$ to
\eqref{eq:syst-grenier}, where $\w(t)=\w_0-Mt$ and
$T=T(\ell)<\w_0/M$. Moreover, up to the choice of a possibly larger value for
$M$ (and consequently a smaller one for $T$), we have the estimates 
$$\vvvert\phi^\eps\vvvert_{\ell+1,T}^2\le
4\|\phi_0^\eps\|_{\H_{\w_0}^{\ell+1}}^2+\|a_0^\eps\|_{\H_{\w_0}^{\ell}}^{4\sigma},\qquad
\vvvert a^\eps\vvvert_{\ell,T}^2\le
2\|a_0^\eps\|_{\H_{\w_0}^{\ell}}^2.$$ 
\end{theorem}
 Unlike in the framework of Sobolev spaces, we do not have tame
  estimates in $\H_{\w}^\ell$. This is the reason why the existence
  time in the above result depends a priori on $\ell$. (In the
  Sobolev case, the existence time for hyperbolic systems in $H^s$ does
  not depend on $s>d/2+1$, thanks to tame estimates.) It is natural to
  consider that the map $\ell\mapsto M(\ell)$ is increasing. In other
  words, $T$ in Theorem~\ref{th:wp} is a decreasing function of $\ell$. 

Our second result states the convergence of the phase and of the
complex amplitude as $\eps\to 0$. 
\begin{theorem}\label{th:conv-obs}
Let $\w_0>0$, $\ell>1$,
$(\phi_0,a_0)\in\H_{\w_0}^{\ell+2}\times\H_{\w_0}^{\ell+1} $ and
$(\phi_0^\eps,a_0^\eps)_{\eps\in(0,1]}$ bounded in
$\H_{\w_0}^{\ell+1}\times\H_{\w_0}^{\ell}$ such that 
\begin{equation*}
  r_0^\eps:= \|\phi_0^\eps-
\phi_0\|_{\H_{\w_0}^{\ell+1}} + \|a_0^\eps-
a_0\|_{\H_{\w_0}^{\ell}} \Tend \eps 0 0.
\end{equation*}
Let $M=M(\ell+1)$ and $T=T(\ell+1)$, as defined as in Theorem~\ref{th:wp}. 
 Then there is an
 $\eps$-independent 
 $C>0$ such that for all $\eps\in (0,1]$, 
$$\vvvert\phi^\eps-\phi\vvvert_{\ell+1,T}+\vvvert
a^\eps-a\vvvert_{\ell,T}\le C\(r_0^\eps+\eps\),$$ 
where $(\phi^\eps,a^\eps)$ denotes the solution to
\eqref{eq:syst-grenier}  and
$(\phi,a)$ is the solution to \eqref{eq:syst-limit}. 
\end{theorem}
The fact that in the above statement, $T(\ell+1) (<T(\ell)$) is considered,
and not simply $T(\ell)$, is reminiscent of the fact that to prove error
estimates in WKB expansions, one has to pay some price in terms of
regularity, even in the linear case when working in Sobolev spaces (see
e.g. \cite[Chapter~1]{CaBook}). 

However,  regarding convergence of the wave functions, the previous
result is not sufficient. Indeed, as fast as $\phi_0^\eps$ and
$a_0^\eps$ may converge as $\eps\to 0$, Theorem \ref{th:conv-obs} at most
guarantees that $\phi^\eps-\phi=\mathcal{O}(\eps)$, which only ensures
that $a^\eps e^{i\phi^\eps/\eps}-ae^{i\phi/\eps}=\mathcal{O}(1)$, due
to the rapid oscillations. However, the above convergence result
suffices to infer the convergence of quadratic observables:
\begin{corollary}\label{cor:quad}
  Under the assumptions of Theorem~\ref{th:conv-obs}, the position and
  momentum densities converge:
  \begin{equation*}
    |u^\eps|^2\Tend \eps 0 |a|^2,\quad \text{and}\quad\IM \(\eps \bar
    u^\eps \d_x 
    u^\eps\) \Tend \eps 0 |a|^2\d_x \phi,\quad\text{in
    }L^\infty([0,T];L^1\cap L^\infty(\R)).
  \end{equation*}
\end{corollary}
 In
order to get a good approximation of the wave function $a^\eps
e^{i\phi^\eps/\eps}$, we have to approximate $\phi^\eps$ up to an
 error which is small compared to $\eps$. It will be done by adding a
 corrective 
term to $(\phi,a)$. For this purpose, we consider the system
obtained by linearizing \eqref{eq:syst-grenier} about
$(\phi,a)$, solution to \eqref{eq:syst-limit},
\begin{equation}\label{eq:syst-grenier-linearized}
  \left\{
    \begin{aligned}
      &\d_t \phi_1 + \left(\partial_x
        \phi+\frac{1}{2}g\(|a|^2\)\right)\partial_x \phi_1 \\
&\qquad + \left(
        g'\(|a|^2\)\partial_x \phi
        +2f'(|a|^2)\right)\RE\left(\overline{a}a_1\right)=0 ,\quad 
      {\phi_1}_{\mid t=0}=\phi_{10},\\ 
& \d_t a_1  + \d_x \phi\d_x a_1+\frac{1}{2}a_1\d_x^2 \phi + \d_x a\d_x \phi_1 
+\frac{1}{2}a\d_x^2 \phi_1  
+\frac{1}{2}\d_x\( g\(|a|^2\)a_1\)\\
&\qquad+\frac{1}{2}\d_x\(2a g'\(|a|^2\)\RE\left(\overline{a}a_1\right)\) =
\frac{i}{2}\d_x^2 a,\qquad{a_1}_{\mid t=0}=a_{10}.
    \end{aligned}
\right.
\end{equation}
We refer to \cite{CaBook} for a discussion on the appearance of these
correctors, and in particular regarding cases where they are trivial
or not. 
Provided $(\phi_0,a_0)\in \H_{\w_0}^{\ell+3}\times\H_{\w_0}^{\ell+2}$
(which implies $(\phi,a)\in X_{\w,T}^{\ell+2}$ according to
Theorem \ref{th:wp}) and $(\phi_{10},a_{10})\in
\H_{\w_0}^{\ell+2}\times\H_{\w_0}^{\ell+1}$, we will see that the solution to
(\ref{eq:syst-grenier-linearized}) belongs to $X_{\w,T}^{\ell+1}$, and
our final result is the following. 
 
\begin{theorem}
\label{th:conv-wf}
Let $\w_0>0$, $\ell>1$,
$(\phi_0,a_0)\in\H_{\w_0}^{\ell+3}\times\H_{\w_0}^{\ell+2}$,
$(\phi_{10},a_{10})\in\H_{\w_0}^{\ell+2}\times\H_{\w_0}^{\ell+1} $ and
$(\phi_0^\eps,a_0^\eps)_{\eps\in(0,1]}$ bounded in
$\H_{\w_0}^{\ell+1}\times\H_{\w_0}^{\ell}$ such that 
\begin{equation*}
  r_1^\eps:=\|\phi_0^\eps- \phi_0-\eps\phi_{10}\|_{\H_{\w_0}^{\ell+1}}
  + \|a_0^\eps- a_0-\eps a_{10}\|_{\H_{\w_0}^{\ell}}=o(\eps)\text{ as
  } \eps \to 0.
\end{equation*}
Then, for $M=M(\ell+2)$ and $T=T(\ell+2)$  as in Theorem~\ref{th:wp},
there is 
an $\eps$-independent  $C>0$ such that for all $\eps\in (0,1]$, 
\begin{equation}\label{eq:est2}
\vvvert\phi^\eps-\phi-\eps \phi_1\vvvert_{\ell+1,T}+\vvvert
a^\eps-a-\eps a_1\vvvert_{\ell,T}\le C\(r_1^\eps +\eps^2\),
\end{equation}
where $(\phi^\eps,a^\eps)$ denotes the solution to
\eqref{eq:syst-grenier}, $(\phi,a)$
is the solution to \eqref{eq:syst-limit}, and $(\phi_1,a_1)$ is the solution to
\eqref{eq:syst-grenier-linearized}. In particular,
\begin{equation*}
  \left\|u^\eps - a
    e^{i\phi_1}e^{i\phi/\eps}\right\|_{L^\infty([0,T];L^2\cap
    L^\infty(\R))}=\O\(\frac{r_1^\eps}{\eps}+\eps\)\Tend \eps 0 0 . 
\end{equation*}
\end{theorem}

\section{Well-posedness}\label{sec:cauchy}
In this section,  $\eps\in [0,1]$ is fixed. Solutions to (\ref{eq:syst-grenier}) are
constructed  as limits of the solutions of the iterative scheme  
\begin{equation}\label{eq:iteration}
  \left\{
    \begin{aligned}
      &\d_t \phi^\eps_{j+1} + \frac{1}{2}\d_x\phi^\eps_j
      \d_x\phi^\eps_{j+1}+\frac{1}{2}g(|a_j^\eps|^2)\d_x\phi^\eps_{j+1} +f(|a_j^\eps|^2) =0,\\  
&      \qquad\qquad \phi^\eps_{j+1\mid t=0}=\phi_0^\eps,\\
& \d_t a^\eps_{j+1} + \d_x\phi^\eps_j  \d_x a^\eps_{j+1}
+\frac{1}{2}a^\eps_{j+1}\d_x^2 \phi^\eps_j +\frac{1}{2}\d_x
\(g\(|a_j^\eps|^2\)\)a^\eps_{j+1} \\ 
&\quad+\frac{1}{2}h\(|a^\eps_j|^2\)\overline{a_j^\eps} a^\eps_{j+1}\d_x a_j^\eps=
\frac{i\eps}{2}\d^2_x a^\eps_{j+1}, \quad a^\eps_{j+1\mid t=0}=a_0^\eps,
\end{aligned}
\right.
\end{equation}
where $h(s)=g(s)/s$. The scheme is initialized with time-independent
$(\phi_0^\eps,a_0^\eps)\in
\H_{\w_0}^{\ell+1}\times\H_{\w_0}^{\ell}\subset X_{\w,T}^\ell$ for any
$T>0$. 
\smallbreak

The scheme is well-defined: if $\ell>1$, for a given $(\phi_j^\eps,a_j^\eps)\in X_{\w,T}^\ell$, (\ref{eq:iteration}) defines $(\phi_{j+1}^\eps,a_{j+1}^\eps)$.
Indeed, in the first equation, $\phi_{j+1}^\eps$
solves a linear transport equation with smooth coefficients, and the 
second equation is equivalent through the relation $v_{j+1}^\eps = a_{j+1}^\eps
e^{i\phi_j^\eps/\eps}$ to the linear Schr\"odinger equation
\begin{align*}
&  i\eps\d_t v_{j+1}^\eps +\frac{\eps^2}{2}\d_x^2 v_{j+1}^\eps\\ 
 &\quad =
  -\(\d_t \phi_j^\eps +\frac{1}{2}(\d_x
  \phi_j^\eps)^2+\frac{i\eps}{2}\d_x \(g\(|a_j^\eps|^2\)\) +\frac{i\eps}{2}h(|a^\eps_j|^2)\overline{a_j^\eps}\d_x a_j^\eps\)v_{j+1}^\eps
\end{align*}
with initial condition
$$\quad v_{j+1\mid t=0}^\eps= a_0^\eps e^{i\phi_0^\eps/\eps}.$$
This is a linear Schr\"odinger equation with a
smooth and bounded external time-dependent potential, for which the
existence of an $L^2$-solution is granted, by perturbative arguments
(the potential is complex-valued). 
\smallbreak

We recall the following lemma, which is proved in \cite{GV01}.
\begin{lemma}\label{lem:tame} Let $m\ge 0$
and $s>1/2$. Then, there is a constant $C>0$, independent of
$\w\ge 0$, such that for every $\psi_1,\psi_2\in \H_\w^{\max(m,s)}$,
\begin{equation}\label{eq:tame}
    \|\psi_1 \psi_2\|_{\H_\w^m}\le C\(
    \|\psi_1\|_{\H_\w^m}\|\psi_2\|_{\H_\w^s} +
    \|\psi_1\|_{\H_\w^s}\|\psi_2\|_{\H_\w^m}\).  
  \end{equation}
\end{lemma}
The following lemma is a toolbox for all the forthcoming analysis.

\begin{lemma}\label{lem:evol-norm}
 Let  $\ell>1$ and $T>0$. Let $(\phi,a)\in X_{\w,T}^\ell$,
 $\tilde{a}\in Y_{\w,T}^{\ell+1}$ and $(\FF,\GG)\in
 L^2([0,T],\H_\w^{\ell+1/2}\times \H_\w^{\ell-1/2})$ such that 
\begin{align}
&\d_t\phi=\FF,& \phi(0)\in\H_{\w_0}^{\ell+1},\label{eq:phi}\\
&\d_t a=\GG+i\theta_1\d_x^2 a+i\theta_2\d_x^2\tilde{a},& a(0)\in\H_{\w_0}^{\ell},\label{eq:a}
\end{align}    
where $\theta_1,\theta_2\in\R$. Then
\begin{align}
&\vvvert\phi\vvvert_{\ell+1,T}^2  \le
  \|\phi(0)\|_{\H_{\w_0}^{\ell+1}}^2+\frac{1}{M}\vvvert\phi\vvvert_{\ell+1,T}\sqrt{2M}\|\FF\|_{L^2_T\H_{\w}^{\ell+1/2}},\label{eq:phiY}\\ 
&\vvvert a\vvvert_{\ell,T}^2  \le
  \|a(0)\|_{\H_{\w_0}^{\ell}}^2+\frac{1}{M}\vvvert
  a\vvvert_{\ell,T}\sqrt{2M}\|\GG\|_{L^2_T\H_{\w}^{\ell-1/2}}+
\frac{|\theta_2|}{2M}\vvvert 
  a\vvvert_{\ell,T}\vvvert \tilde{a}\vvvert_{\ell+1,T}.\label{eq:aY} 
\end{align}
Moreover, there exists $C>0$ (that depends only on $\ell$) such that
\begin{itemize}
\item If $\FF=\d_x\psi_1\d_x\psi_2$ with $\psi_1,\psi_2\in
  Y_{\ell+1,T}$, then 
\begin{equation}\label{eq:mu1} 
\sqrt{2M}\|\FF\|_{L^2_T\H_{\w}^{\ell+1/2}}\le C\vvvert
\psi_1\vvvert_{\ell+1,T}\vvvert \psi_2\vvvert_{\ell+1,T} .
\end{equation}
\item If
  $\FF=\left(\overset{2n}{\underset{j=1}{\prod}}b_j\right)\d_x\psi$
  with $\psi\in Y_{\ell+1,T}$ and $b_j\in Y_{\ell,T}$ for all $j$,
  then 
\begin{equation}\label{eq:mu2}
\sqrt{2M}\|\FF\|_{L^2_T\H_{\w}^{\ell+1/2}}\le C\left(\overset{2n}{\underset{j=1}{\prod}} \vvvert b_j\vvvert_{\ell,T}\right) \vvvert\psi\vvvert_{\ell+1,T}.
\end{equation}
\item If $\FF=\overset{2n}{\underset{j=1}{\prod}}b_j$ with $b_j\in Y_{\ell,T}$ for all $j$, then
\begin{equation}\label{eq:mu3}
\sqrt{2M}\|\FF\|_{L^2_T\H_{\w}^{\ell+1/2}}\le C\overset{2n}{\underset{j=1}{\prod}} \vvvert b_j\vvvert_{\ell,T}.
\end{equation}
\item If $\GG=\d_x\psi\d_x b$ with $\psi\in Y_{\ell+1,T}$ and $b\in Y_{\ell,T}$, then
\begin{equation}\label{eq:nu1}
\sqrt{2M}\|\GG\|_{L^2_T\H_{\w}^{\ell-1/2}}\le C\vvvert \psi\vvvert_{\ell+1,T}\vvvert b\vvvert_{\ell,T}.
\end{equation}
\item If $\GG=b\d_x^2\psi $ with $\psi\in Y_{\ell+1,T}$ and $b\in Y_{\ell,T}$, then
\begin{equation}\label{eq:nu2}
\sqrt{2M}\|\GG\|_{L^2_T\H_{\w}^{\ell-1/2}}\le C\vvvert \psi\vvvert_{\ell+1,T}\vvvert b\vvvert_{\ell,T}.
\end{equation}
\item If $\GG=\left(\overset{2n}{\underset{j=1}{\prod}}b_j\right)\d_x b$ with $b, b_j\in Y_{\ell,T}$ for all $j$, then
\begin{equation}\label{eq:nu3}
\sqrt{2M}\|\GG\|_{L^2_T\H_{\w}^{\ell-1/2}}\le C\left(\overset{2n}{\underset{j=1}{\prod}} \vvvert b_j\vvvert_{\ell,T}\right) \vvvert b\vvvert_{\ell,T}.
\end{equation}
\end{itemize}
\end{lemma}

\begin{remark}
In the proof given below, the assumption $\ell>1$ is used only to
establish the last three estimates on $\GG$. The rest of the proof only
requires the assumption $\ell>1/2$. Actually, even the estimates on
$\GG$ can be proved under the condition $\ell>1/2$, thanks to a
refined version of Lemma \ref{lem:tame} (see \cite{CaGa-p}). However,
since it is not useful in the sequel to sharpen this assumption, we
choose to make the stronger assumption $\ell>1$ for the sake of
conciseness. 
\end{remark}

\begin{proof}[Proof of Lemma~\ref{lem:evol-norm}]
For fixed $t\in[0,T]$, \eqref{eq:evol-norm} yields
\begin{align*}
\frac{d}{dt} \|\phi\|_{\H_{\w(t)}^{\ell+1}}^2 + 2M
   \|\phi\|_{\H_{\w(t)}^{\ell+3/2}}^2&   = 2\RE\<\phi,\FF\>_{\H_{\w(t)}^{\ell+1}} \le 2\|\phi\|_{\H_{\w(t)}^{\ell+3/2}}\|\FF\|_{\H_{\w(t)}^{\ell+1/2}}.
   \end{align*}
By integration and Cauchy-Schwarz inequality in time, we get, for
every $t\in[0,T]$, 
\begin{align*}
\|\phi(t)\|_{\H_{\w(t)}^{\ell+1}}^2+2M\int_0^t
   \|\phi(\tau)\|_{\H_{\w(\tau)}^{\ell+3/2}}^2d\tau\le \|\phi(0)\|_{\H_{\w_0}^{\ell+1}}^2+2\|\phi\|_{L^2_T\H_\w^{\ell+3/2}}\|\FF\|_{L^2_T\H_{\w}^{\ell+1/2}},
\end{align*}
hence (\ref{eq:phiY}). The proof of (\ref{eq:aY}) is similar, taking into account that $$\RE\< i\theta_1\d_x^2 a,a\>_{\H_{\w(t)}^\ell}=0,$$ 
since $\theta_1\in \R$, and that 
$$\left|\<i\theta_2\d_x^2\tilde{a},a\>_{\H_{\w(t)}^\ell}\right|=|\theta_2|
\left|\int \<\xi\>^{2\ell}e^{2w\<\xi\>} \xi^2 \F(\tilde a
  )\overline{\F(a)}d\xi\right| 
\le|\theta_2|\|a\|_{\H_{\w(t)}^{\ell+1/2}}\|\tilde{a}\|_{\H_{\w(t)}^{\ell+3/2}}.$$
In order to prove (\ref{eq:mu1}), let us first use (\ref{eq:tame}) with $m=\ell+1/2$ and $s=\ell$: for every $t\in[0,T]$, we obtain
\begin{align*}
\|\FF\|_{\H_{\w(t)}^{\ell+1/2}}=\|\d_x\psi_1\d_x\psi_2\|_{\H_{\w(t)}^{\ell+1/2}}
\lesssim\left(\|\psi_1\|_{\H_{\w(t)}^{\ell+3/2}}\|\psi_2\|_{\H_{\w(t)}^{\ell+1}}+
\|\psi_1\|_{\H_{\w(t)}^{\ell+1}}\|\psi_2\|_{\H_{\w(t)}^{\ell+3/2}}\right). 
\end{align*}
Taking the $L^2$ norm in time in the last estimate, we get
\begin{equation*}
\|\FF\|_{L^2_T\H_{\w}^{\ell+1/2}}\lesssim
\left(\|\psi_1\|_{L^2_T\H_{\w}^{\ell+3/2}}\|\psi_2\|_{L^\infty_T\H_{\w}^{\ell+1}}+
\|\psi_1\|_{L^\infty_T\H_{\w}^{\ell+1}}\|\psi_2\|_{L^2_T\H_{\w}^{\ell+3/2}}\right), 
\end{equation*}
hence (\ref{eq:mu1}). The proofs of (\ref{eq:mu2}) and (\ref{eq:mu3})
are similar, thanks to multiple uses of (\ref{eq:tame}) with
$m=\ell+1/2$ and $s=\ell$. The proofs of (\ref{eq:nu1}),
(\ref{eq:nu2}) and (\ref{eq:nu3}) are also similar, except that
(\ref{eq:tame}) is now applied with $m=s=\ell-1/2>1/2$.  
\end{proof}

\begin{proof}[Proof of Theorem~\ref{th:wp}]
In view of the equation satisfied by $\phi_{j+1}^\eps$ in
(\ref{eq:iteration}), Lemma~\ref{lem:evol-norm} yields
\begin{align*}
  \vvvert  \phi_{j+1}^\eps\vvvert_{\ell+1,T}^2 &\le \|
  \phi_0^\eps\|_{\H_{\w_0}^{\ell+1}}^2 +\frac{C}{M} 
\vvvert \phi_{j+1}^\eps\vvvert_{\ell+1,T}^2 
\vvvert \phi_{j}^\eps\vvvert_{\ell+1,T}
+\frac{C}{M} 
\vvvert \phi_{j+1}^\eps\vvvert_{\ell+1,T}
\vvvert a_{j}^\eps\vvvert_{\ell,T}^{2\si}\\
&\quad +\frac{C}{M} \vvvert \phi_{j+1}^\eps\vvvert_{\ell+1,T}^2\vvvert a_{j}^\eps\vvvert_{\ell,T}^{2\gamma}.
\end{align*}
As for $a_{j+1}^\eps$, we obtain in a similar way
\begin{align*}
  \vvvert  a_{j+1}^\eps\vvvert_{\ell,T}^2 &\le \|
 a_0^\eps\|_{\H_{\w_0}^\ell}^2 +\frac{C}{M} 
\vvvert a_{j+1}^\eps\vvvert_{\ell,T}^2 
\vvvert \phi_{j}^\eps\vvvert_{\ell+1,T}+\frac{C}{M} 
\vvvert a_{j+1}^\eps\vvvert_{\ell,T}^2 
\vvvert a_{j}^\eps\vvvert_{\ell,T}^{2\gamma}.
\end{align*}
Under the condition
\begin{equation}\label{eq:rec}
\frac{C}{M}\vvvert\phi_j^\eps\vvvert_{\ell+1,T}\le\frac{1}{4},\qquad \frac{C}{M}\vvvert a_j^\eps\vvvert_{\ell,T}^{2\gamma}\le\frac{1}{4},
\end{equation}
we infer
\begin{align}\label{eq-hyp-rec-j+1}
\frac{1}{4}\vvvert\phi_{j+1}^\eps\vvvert_{\ell+1,T}^2\le\|\phi_0^\eps\|_{\H_{\w_0}^{\ell+1}}^2+\frac{C^2}{M^2}\vvvert a_j^\eps\vvvert_{\ell,T}^{4\sigma},\qquad
\frac{1}{2}\vvvert a_{j+1}^\eps\vvvert_{\ell,T}^2\le\|a_0^\eps\|_{\H_{\w_0}^{\ell}}^2.
\end{align}
{\bf First step: boundedness of the sequence.}
We show by induction that, provided $M$ is sufficiently
large, we can construct a sequence $(\phi_j^\eps, a_j^\eps)_{j\in\N}$
such that for every $j\in\N$,  
\begin{align}\label{eq:hyp-rec}
\vvvert\phi_{j}^\eps\vvvert_{\ell+1,T}^2\le4\|\phi_0^\eps\|_{\H_{\w_0}^{\ell+1}}^2+\frac{4C^2}{M^2}\(2\|a_0^\eps\|_{\H_{\w_0}^{\ell}}^2\)^{2\sigma},\qquad
\vvvert a_{j}^\eps\vvvert_{\ell,T}^2\le 2\|a_0^\eps\|_{\H_{\w_0}^{\ell}}^2.
\end{align}
For that purpose, we choose $M$ sufficiently large such that
(\ref{eq:rec}) holds for $j=0$ and such that  
\begin{align}\label{eq:cond-M}
4\|\phi_0^\eps\|_{\H_{\w_0}^{\ell+1}}^2+\frac{4C^2}{M^2}\(2\|a_0^\eps\|_{\H_{\w_0}^{\ell}}^2\)^{2\sigma}\le \frac{M^2}{16C^2},\qquad (2\|a_0^\eps\|_{\H_{\w_0}^{\ell}}^2)^\gamma\le \frac{M}{4C}.
\end{align}
Then, (\ref{eq:hyp-rec}) holds for $j=0$, since with
$(\phi_0^\eps,a_0^\eps)(t,x)=(\phi_0^\eps,a_0^\eps)(x)$ independent of
time,  it is
easy to check that $\vvvert
\phi_0^\eps\vvvert_{\ell+1,T}= \|\phi_0^\eps\|_{\H_{\w_0}^{\ell+1}}$
and $\vvvert
a_0^\eps\vvvert_{\ell,T}=\|a_0^\eps\|_{\H_{\w_0}^{\ell}}$. Let $j\ge
0$ and assume that (\ref{eq:hyp-rec}) holds. Then (\ref{eq:hyp-rec})
and (\ref{eq:cond-M}) ensure that the condition (\ref{eq:rec}) is
satisfied, and therefore (\ref{eq-hyp-rec-j+1}) holds, from which we
infer easily that (\ref{eq:hyp-rec}) is true for $j$ replaced by
$j+1$. 
\smallbreak

\noindent {\bf Second step: convergence.}
For $j\ge 1$, we set
$\delta\phi_j^\eps=\phi_j^\eps-\phi_{j-1}^\eps$, and 
$\delta a_j^\eps=a_j^\eps-a_{j-1}^\eps$. Then, for every $j\ge
1$, we have 
\begin{align*}
  &\d_t \delta\phi_{j+1}^\eps+ \frac{1}{2}\(\d_x \phi^\eps_j
      \d_x \delta \phi^\eps_{j+1}  + \d_x \delta \phi^\eps_j
      \d_x\phi^\eps_{j}\)+\frac{1}{2}g\(|a_j^\eps|^2\)\d_x\delta\phi_{j+1}^\eps\\
      &\qquad+\frac{1}{2}\(g\(|a_j^\eps|^2\)-g\(|a_{j-1}^\eps|^2\)\)\d_x\phi^\eps_{j} +f(|a_j^\eps|^2)-f(|a_{j-1}^\eps|^2)=0.
\end{align*}
Taking into account that
$$|a^\eps_j|^{2\gamma}-|a^\eps_{j-1}|^{2\gamma}=
\sum_{k=0}^{\gamma-1}(a^\eps_{j-1})^k\delta 
a^\eps_j(a^\eps_j)^{\gamma-1-k}\overline{a^\eps_j}^\gamma+
\sum_{k=0}^{\gamma-1}(a^\eps_{j-1})^\gamma
\overline{a^\eps_{j-1}}^k\overline{\delta
  a^\eps_j}\overline{a^\eps_j}^{\gamma-1-k}, $$
and that the same equality holds for $\gamma$ replaced by $\sigma$, Lemma \ref{lem:evol-norm} and (\ref{eq:hyp-rec}) imply
\begin{align*}
 \vvvert  \delta\phi_{j+1}^\eps\vvvert_{\ell+1,T}^2 &\le \frac{K}{M}\(  \vvvert  \delta\phi_{j+1}^\eps\vvvert_{\ell+1,T}^2 
+\vvvert  \delta\phi_{j}^\eps\vvvert_{\ell+1,T}^2 +
 \vvvert  \delta a_{j}^\eps\vvvert_{\ell,T}^2 \)
\end{align*}
for some $K>0$, which does not depend on $\eps$ provided
$(\phi_0^\eps, a_0^\eps)_{\eps\in [0,1]}$ is uniformly bounded in
$\H_{\w_0}^{\ell+1}\times \H_{\w_0}^{\ell}$. 
Thus, for $M$ large enough,
\begin{equation*}
    \vvvert  \delta\phi_{j+1}^\eps\vvvert_{\ell+1,T}^2  \le 
\frac{2K}{M}\(  \vvvert  \delta\phi_{j}^\eps\vvvert_{\ell+1,T}^2 +
 \vvvert  \delta a_{j}^\eps\vvvert_{\ell,T}^2 \).
\end{equation*}
Similarly, $\delta a_{j+1}^\eps$ solves
\begin{align*}
&  \d_t \delta a_{j+1}^\eps + \d_x \phi_j^\eps \d_x \delta
  a_{j+1}^\eps + \d_x \delta \phi_j^\eps \d_x a_j^\eps +
  \frac{1}{2}\delta a_{j+1}^\eps\d_x^2 \phi_j^\eps +
  \frac{1}{2}a_j^\eps\d_x^2 \delta \phi_j^\eps \\
&\quad  +\frac{1}{2}\d_x \(g\(|a_j^\eps|^2\)\)\delta a^\eps_{j+1}+
  \frac{1}{2}\d_x\(g\(|a_j^\eps|^2\)-g\(|a_{j-1}^\eps|^2\)\)a_j^\eps\\
 &\quad  +\frac{1}{2} h\(|a^\eps_j|^2\)\d_x a_j^\eps\overline{a_j^\eps} \delta a^\eps_{j+1}
 +\frac{1}{2} h\(|a^\eps_j|^2\)\d_x a_j^\eps\overline{\delta a_j^\eps}  a^\eps_{j}\\
&\quad +\frac{1}{2} h\(|a^\eps_j|^2\)\d_x \delta a_j^\eps\overline{ a_{j-1}^\eps} a^\eps_{j}
 +\frac{1}{2} \(h\(|a^\eps_j|^2\)-h\(|a^\eps_{j-1}|^2\)\)\d_x  a_{j-1}^\eps\overline{ a_{j-1}^\eps} a^\eps_{j}\\
&\quad  = i\frac{\eps}{2}\d_x^2
  \delta a_{j+1}^\eps, 
\end{align*}
so Lemma \ref{lem:evol-norm} and (\ref{eq:hyp-rec}) yield
\begin{equation*}
    \vvvert  \delta a_{j+1}^\eps\vvvert_{\ell,T}^2  \le 
\frac{2K}{M}\(  
\vvvert  \delta\phi_{j}^\eps\vvvert_{\ell+1,T}^2 +
 \vvvert  \delta a_{j}^\eps\vvvert_{\ell,T}^2 \).
\end{equation*}
We conclude  that provided $\ell>1$, possibly
increasing $M$, $(\phi_j^\eps,a_j^\eps)$ converges geometrically 
in $X_{\w,T}^\ell$
as $j\to \infty$. Uniqueness of the solution $(\phi^\eps, a^\eps)$ to
(\ref{eq:syst-grenier}) follows from the same kind of estimates as the
ones which prove the convergence. 
\end{proof}

\section{First order approximation} 
\begin{proof}[Proof of Theorem~\ref{th:conv-obs}]
Next, assume that
$(\phi_0,a_0)\in\H_{\w_0}^{\ell+2}\times\H_{\w_0}^{\ell+1} $. Then, in
view of Theorem~\ref{th:wp},
the solution $(\phi,a)$ to  \eqref{eq:syst-limit} belongs to
$X_{\w,T}^{\ell+1}$. Given $\eps>0$, if $(\phi_0^\eps,a_0^\eps)\in
\H_{\w_0}^{\ell+1}\times\H_{\w_0}^{\ell}$, we denote by
$(\phi^\eps,a^\eps)$ the solution to (\ref{eq:syst-grenier}). We also denote
$(\delta\phi^\eps, \delta a^\eps)=(\phi^\eps-\phi, a^\eps-a)$.  Then,
in the same fashion as above, we have 
\begin{align*}
  &\d_t \delta\phi^\eps+ \frac{1}{2}\(\d_x \phi^\eps
      \d_x \delta \phi^\eps + \d_x \delta \phi^\eps
      \d_x\phi\)+\frac{1}{2}g\(|a^\eps|^2\)\d_x\delta\phi^\eps\\
      &\qquad+\frac{1}{2}\(g\(|a^\eps|^2\)-g\(|a|^2\)\)\d_x\phi +f(|a^\eps|^2)-f(|a|^2)=0,
\end{align*}
and
\begin{align*}
&  \d_t \delta a^\eps + \d_x \phi^\eps \d_x \delta
  a^\eps + \d_x \delta \phi^\eps \d_x a +
  \frac{1}{2}\delta a^\eps\d_x^2 \phi^\eps +
  \frac{1}{2}a\d_x^2 \delta \phi^\eps \\
&\quad  +\frac{1}{2}\d_x \(g\(|a^\eps|^2\)\)\delta a^\eps+
  \frac{1}{2}\d_x\(g\(|a^\eps|^2\)-g\(|a|^2\)\)a\\
 &\quad  +\frac{1}{2} h\(|a^\eps|^2\)\d_x a^\eps\overline{a^\eps} \delta a^\eps
 +\frac{1}{2} h\(|a^\eps|^2\)\d_x a^\eps\overline{\delta a^\eps}  a\\
&\quad +\frac{1}{2} h\(|a^\eps|^2\)\d_x \delta a^\eps|a|^2
 +\frac{1}{2} \(h\(|a^\eps|^2\)-h\(|a|^2\)\)\d_x  a|a|^2
  = i\frac{\eps}{2}\d_x^2 \delta a^\eps+i\frac{\eps}{2}\d_x^2 a.
\end{align*}
For some new constant $K$, Lemma \ref{lem:evol-norm} and Theorem \ref{th:wp} imply, for $M$ large enough, 
\begin{equation*}
    \vvvert  \delta\phi^\eps\vvvert_{\ell+1,T}^2  \le 
K\|\phi_0^\eps-\phi_0\|_{\H_{\w_0}^{\ell+1}}^2+\frac{K}{M}
 \vvvert  \delta a^\eps\vvvert_{\ell,T}^2,
\end{equation*}
and
\begin{equation*}
    \vvvert  \delta a^\eps\vvvert_{\ell,T}^2  \le 
K\|a_0^\eps-a_0\|_{\H_{\w_0}^\ell}^2+\frac{K}{M} 
\vvvert  \delta\phi^\eps\vvvert_{\ell+1,T}^2+\frac{K}{M}\eps \vvvert\delta a^\eps\vvvert_{\ell,T}\vvvert a\vvvert_{\ell+1,T}.
\end{equation*}
Possibly increasing the value of $M$ and adding the last two inequalities, we deduce
\begin{equation*}
    \vvvert  \delta\phi^\eps\vvvert_{\ell+1,T}^2+\vvvert  \delta a^\eps\vvvert_{\ell,T}^2\le C\|\phi_0^\eps-\phi_0\|_{\H_{\w_0}^{\ell+1}}^2+C\|a_0^\eps-a_0\|_{\H_{\w_0}^\ell}^2+C\eps^2,
\end{equation*}
hence Theorem \ref{th:conv-obs}. As for the choice of $M$, a careful examination of the previous inequalities shows that aside from the assumption $M\ge M(\ell+1)$, which enables to estimate the source term, $M$ can be chosen as in Theorem \ref{th:wp}, namely such that $M\ge M(\ell)$.
\end{proof}
 
\begin{proof}[Proof of Corollary~\ref{cor:quad}]
Notice that,  provided  $\w\ge 0$,
\begin{equation}\label{eq:obvious}
  \|\psi\|_{H^\ell(\R)}\le \|\psi\|_{\H_\w^\ell}.
\end{equation}
In particular, Sobolev embedding yields, for $\ell>1/2$,
\begin{equation*}
  \|\psi\|_{L^\infty(\R)}\le C \|\psi\|_{\H_\w^\ell},
\end{equation*}
where $C$ is independent of $\w \ge 0$. 
With these remarks in mind, the $L^1$ estimates of
Corollary~\ref{cor:quad} follow from Theorem~\ref{th:conv-obs} and 
Cauchy-Schwarz inequality, since
\begin{equation*}
  \left\| |u^\eps|^2-|a|^2\right\|_{L^\infty_TL^1} =  \left\|
    |a^\eps|^2-|a|^2\right\|_{L^\infty_TL^1}  \le
  \|a^\eps+a\|_{L^\infty_T L^2} \|\delta a^\eps\|_{L^\infty_TL^2},
\end{equation*}
and
\begin{align*}
  \|\IM \(\eps \bar u^\eps \d_x u^\eps\)
  -|a|^2\d_x\phi\|_{L^\infty_TL^1} &\le \eps\|\IM \bar a^\eps \d_x
  a^\eps\|_{L^\infty_TL^1}  + \||a^\eps|^2\d_x\phi^\eps-|a|^2\d_x
  \phi\|_{L^\infty_TL^1}  \\
&\le \eps \|a^\eps\|^2_{L^\infty_TH^1}
  +\|a^\eps+a\|_{L^\infty_T L^2} \|\delta a^\eps\|_{L^\infty_TL^2}
  \|\d_x \phi\|_{L^\infty_TL^\infty} \\
&\quad+ \|a^\eps\|_{L^\infty_TL^\infty}
\|a^\eps\|_{L^\infty_TL^2}\|\delta \phi^\eps\|_{L^\infty_TH^1}.
\end{align*}
The $L^\infty$ estimates in space follow by replacing $L^1$ and $L^2$
by $L^\infty$ in the above inequalities, and using Sobolev embedding
again. 
\end{proof}

\section{Convergence of the wave function}
\begin{proof}[Proof of Theorem~\ref{th:conv-wf}]
Let $\ell>1$, and
$(\phi_0,a_0)\in\H_{\w_0}^{\ell+2}\times\H_{\w_0}^{\ell+1}$. Theorem~\ref{th:wp}
yields a unique solution
$(\phi,a)\in X_{\w,T}^{\ell+1}$  to \eqref{eq:syst-limit}.
\smallbreak

Let  $(\phi_{10},a_{10})\in \mathcal{H}_{\w_0}^{\ell+1}\times
\mathcal{H}_{\w_0}^{\ell}$. 
Like in Section~\ref{sec:cauchy}, we note that
\eqref{eq:syst-grenier-linearized} is a system of linear transport
equations whose coefficients are smooth functions. The general theory
of transport equations (see e.g. \cite[Section~3]{BCD11}) then shows that
\eqref{eq:syst-grenier-linearized} has a unique solution
$(\phi_1,a_1)\in C([0,T],L^2\times L^2)$. We already know by this
argument that the solution is actually more regular (in terms of
Sobolev regularity), but we shall
directly use a priori estimates in $\H^\ell_\w$ spaces. Indeed,
Lemma \ref{lem:evol-norm} implies that $(\phi_1,a_1)\in X_{\w,T}^\ell$ with
\begin{align*}
&\vvvert\phi_1\vvvert_{\ell+1,T}^2\le \|\phi_{10}\|_{\mathcal{H}_{\w_0}^{\ell+1}}^2
+\frac{C}{M}\vvvert\phi_1\vvvert_{\ell+1,T}^2\vvvert\phi\vvvert_{\ell+1,T}+\frac{C}{M}\vvvert\phi_1\vvvert_{\ell+1,T}^2\vvvert a\vvvert_{\ell,T}^{2\gamma}\\
&\qquad+\frac{C}{M}\vvvert\phi_1\vvvert_{\ell+1,T}\vvvert\phi\vvvert_{\ell+1,T}\vvvert a\vvvert_{\ell,T}^{2\gamma-1}\vvvert a_1\vvvert_{\ell,T}+\frac{C}{M}\vvvert\phi_1\vvvert_{\ell+1,T}\vvvert a\vvvert_{\ell,T}^{2\sigma-1}\vvvert 
a_1\vvvert_{\ell,T},
\end{align*}
along with
\begin{align*}
\vvvert a_1\vvvert_{\ell,T}^2&\le \|a_{10}\|_{\mathcal{H}_{\w_0}^{\ell}}^2+\frac{C}{M}\vvvert a_1\vvvert_{\ell,T}\vvvert a\vvvert_{\ell,T}\vvvert\phi_1\vvvert_{\ell+1,T}+\frac{C}{M}\vvvert a_1\vvvert_{\ell,T}^2\vvvert\phi\vvvert_{\ell+1,T}\\
&\qquad+\frac{C}{M}\vvvert a_1\vvvert_{\ell,T}^2\vvvert a\vvvert_{\ell,T}^{2\gamma}+\frac{C}{M}\vvvert a_1\vvvert_{\ell,T}\vvvert a\vvvert_{\ell+1,T},
\end{align*}
for some $C>0$. 

Let $\ell>1$. For $(\phi_0,a_0)\in \H_{\w_0}^{\ell+3}\times
\H_{\w_0}^{\ell+2}$ and  $(\phi_{10},a_{10})\in
\H_{\w_0}^{\ell+2}\times \H_{\w_0}^{\ell+1}$, we consider:
\begin{itemize}
\item $(\phi,a)\in X_{\w,T}^{\ell+2}$ the solution to \eqref{eq:syst-limit}.
\item $(\phi_1,a_1)\in X_{\w,T}^{\ell+1}$ the solution to
  \eqref{eq:syst-grenier-linearized}.
\item $(\phi_{\rm app}^\eps,a_{\rm app}^\eps)=(\phi,a)+\eps(\phi_1,a_1)$.
\item $(\phi^\eps,a^\eps)\in X_{\w,T}^\ell$ the solution to
  \eqref{eq:syst-grenier}.
\end{itemize}
We assume that $\|\phi_0^\eps- \phi_0-\eps\phi_{10}\|_{\H_{\w_0}^{\ell+1}}=o(\eps)$ and $\|a_0^\eps-a_0-\eps a_{10}\|_{\H_{\w_0}^\ell}=o(\eps)$.
Set 
\begin{align*}
\delta\phi_1^\eps &=\phi^\eps-\phi_{\rm app}^\eps=\phi^\eps-\phi-\eps\phi_1=\delta\phi^\eps-\eps\phi_1,\\
\delta a_1^\eps &=a^\eps-a_{\rm app}^\eps=a^\eps-a-\eps a_1=\delta
                  a^\eps-\eps a_1.
\end{align*}
The equation satisfied by $\delta\phi_1^\eps$ writes
\begin{align*}
  &\d_t \delta\phi_1^\eps+ \d_x\phi\d_x \delta \phi_1^\eps + \frac{1}{2}|\d_x\delta \phi^\eps|^2\\
&     +\frac{1}{2}\left(g\(|a^\eps|^2\)-g\(|a|^2\)-2g'\(|a|^2\)\RE(\overline{a}\eps a_1)\right)\d_x\phi\\
&+\frac{1}{2}\(g\(|a^\eps|^2\)-g\(|a|^2\)\)\eps\d_x\phi_1+\frac{1}{2}g\(|a^\eps|^2\)\d_x\delta\phi_1^\eps\\
     & +f(|a^\eps|^2)-f(|a|^2)-2f'(|a|^2)\RE(\overline{a}\eps a_1)=0
\end{align*}
Moreover, Taylor's formula yields
\begin{align}\label{taylor}
&g(|a^\eps|^2)-g(|a|^2)-2g'(|a|^2)\RE(\overline{a}\eps a_1)\\
\nonumber&\qquad=2g'(|a|^2)\RE(\overline{a}\delta a_1^\eps)+4\RE\left(\overline{a}\delta a^\eps\right)^2\int_0^1(1-s)g''(|a+s\delta a^\eps|^2)ds,
\end{align}
and the same identity holds for $g$ replaced by $f$. Thus, taking into
account Theorem~\ref{th:wp}, which implies
$\vvvert\phi^\eps\vvvert_{\ell+1,T},\vvvert
a^\eps\vvvert_{\ell,T}=\O(1)$, and Theorem \ref{th:conv-obs}, which
implies $\vvvert\delta\phi^\eps\vvvert_{\ell+1,T},\vvvert \delta
a^\eps\vvvert_{\ell,T}=\O(\eps)$, it follows from Lemma
\ref{lem:evol-norm} that  
\begin{align*}
   \vvvert  \delta\phi_{1}^\eps\vvvert_{\ell+1,T}^2 
&\le \|\phi_0^\eps-\phi_0-\eps\phi_{10}\|_{\H_{\w_0}^\ell}
  ^2+\frac{C}{M}\vvvert\delta\phi_1^\eps\vvvert_{\ell+1,T}
  \left(\eps^2+\vvvert\delta\phi_1^\eps\vvvert_{\ell+1,T} 
+\vvvert\delta a_1^\eps\vvvert_{\ell,T}\right).
\end{align*}
 We deduce
\begin{align}\label{eq:df1}
\vvvert  \delta\phi_{1}^\eps\vvvert_{\ell+1,T}^2 
&\le C\|\phi_0^\eps-\phi_0-\eps\phi_{10}\|_{\H_{\w_0}^\ell}  ^2+\frac{C}{M}\eps^4+\frac{C}{M}\vvvert\delta a_1^\eps\vvvert_{\ell,T}^2.
\end{align}
Similarly, $\delta a_1^\eps$ solves
\begin{align*}
&\d_t \delta a_1^\eps+\d_x\phi\d_x\delta a_1^\eps+\d_x\delta\phi_1^\eps\d_x a+\d_x\delta\phi^\eps\d_x \delta a^\eps\\
&\quad+\frac{1}{2}a\d_x^2\delta\phi_1^\eps+\frac{1}{2}\delta a_1^\eps\d_x^2\phi+\frac{1}{2}\delta a^\eps\d_x^2\delta\phi^\eps\\
&\quad+\d_x\left[\left(g(|a^\eps|^2)-g(|a|^2)-2\eps g'(|a|^2)\RE(\overline{a}a_1)\right)a\right]\\
&\quad+\d_x\left[\left(g(|a^\eps|^2)-g(|a|^2)\right)\eps a_1\right]+\d_x\left[g(|a^\eps|^2)\delta a_1^\eps\right]=\frac{i\eps^2}{2}\d_x^2a_1+\frac{i\eps}{2} \d_x^2\delta a_1^\eps.
\end{align*}
From (\ref{taylor}),  Theorems~\ref{th:wp} and \ref{th:conv-obs}, and
Lemma~\ref{lem:evol-norm}, we deduce
\begin{align}\label{eq:da1}
\vvvert\delta a_1^\eps\vvvert_{\ell,T}^2\leq C\|a_0^\eps-a_0-\eps
  a_{10}\|_{\H_{\w_0}^\ell}^2+\frac{C}{M}\eps^4+\frac{C}{M}\vvvert\delta
  \phi_1^\eps\vvvert_{\ell+1,T}^2. 
\end{align}
Adding (\ref{eq:df1}) and (\ref{eq:da1}), \eqref{eq:est2} follows. Like in the proof of Theorem \ref{th:conv-obs}, a
careful examination of the inequalities that we have used shows that
all the above estimates are valid provided that we assume $M\ge M(\ell)$, the
constant provided by Theorem~\ref{th:wp}, and also $M\ge \max(M(\ell+1),M(\ell+2))$ in order to estimate the source terms. 
\smallbreak

To
complete the proof of 
Theorem \ref{th:conv-wf}, consider the point-wise estimate
\begin{align*}
  \left|a^\eps e^{i\phi^\eps/\eps}-a e^{i\phi_1}e^{i\phi/\eps}
  \right|&\le \left|a^\eps-a \right|
+\left|a^\eps \right|
\left|e^{i\phi^\eps/\eps}-e^{i(\phi+\eps\phi_1)/\eps} \right|\\
&\le \left|a^\eps-a \right|
+\left|a^\eps \right|
\left|2\sin\(\frac{\phi^\eps-\phi-\eps\phi_1 }{2\eps}\)\right|\\
&\le \left|\delta a^\eps \right|
+\frac{1}{\eps}\left|a^\eps \right|
\left|\delta \phi_1^\eps\right|.
\end{align*}
We then conclude like in the proof of Corollary~\ref{cor:quad}, by
using Cauchy-Schwarz inequality, \eqref{eq:obvious}, and Sobolev
embedding.  
\end{proof}
\begin{remark}
  The last step in the above proof relies on the estimate
  $\delta\phi_1^\eps=o(\eps)$ in suitable spaces. Regarding the error
  estimate on $a^\eps$, only $\delta a^\eps$ appears. Recall however
  that $(\delta \phi_1^\eps,\delta a_1^\eps)$ solves a \emph{coupled}
  system, so it is necessary to show that $\delta a_1^\eps=o(\eps)$
  too (see also \cite{CaBook}). 
\end{remark}

\subsection*{Acknowledgements}
The authors are grateful to Christof Sparber for attracting their
attention on this problem.

\bibliographystyle{siam}
\bibliography{biblio}

\begin{thebibliography}{10}

\bibitem{AblowitzClarkson}
{\sc M.~J. Ablowitz and P.~A. Clarkson}, {\em Solitons, nonlinear evolution
  equations and inverse scattering}, vol.~149 of London Mathematical Society
  Lecture Note Series, Cambridge University Press, Cambridge, 1991.

\bibitem{ACARMA}
{\sc T.~Alazard and R.~Carles}, {\em Supercritical geometric optics for
  nonlinear {S}chr\"odinger equations}, Arch. Ration. Mech. Anal., 194 (2009),
  pp.~315--347.

\bibitem{BCD11}
{\sc H.~Bahouri, J.-Y. Chemin, and R.~Danchin}, {\em Fourier analysis and
  nonlinear partial differential equations}, vol.~343 of Grundlehren der
  Mathematischen Wissenschaften [Fundamental Principles of Mathematical
  Sciences], Springer, Heidelberg, 2011.

\bibitem{CaBook}
{\sc R.~Carles}, {\em Semi-classical analysis for nonlinear {S}chr\"odinger
  equations}, World Scientific Publishing Co. Pte. Ltd., Hackensack, NJ, 2008.

\bibitem{CaDaSa12}
{\sc R.~Carles, R.~Danchin, and J.-C. Saut}, {\em Madelung,
  {G}ross-{P}itaevskii and {K}orteweg}, Nonlinearity, 25 (2012),
  pp.~2843--2873.

\bibitem{CaGa-p}
{\sc R.~Carles and C.~Gallo}, {\em On {F}ourier time-splitting methods for
  nonlinear {S}chr\"odinger equations in the semi-classical limit {II}.
  {A}nalytic regularity}, Numer. Math.
\newblock to appear. Archived at
  \url{https://hal.archives-ouvertes.fr/hal-01271907}.

\bibitem{ChironRousset}
{\sc D.~Chiron and F.~Rousset}, {\em Geometric optics and boundary layers for
  nonlinear {S}chr\"odinger equations}, Comm. Math. Phys., 288 (2009),
  pp.~503--546.

\bibitem{DesjardinsLin}
{\sc B.~Desjardins and C.-K. Lin}, {\em On the semiclassical limit of the
  general modified {NLS} equation}, J. Math. Anal. Appl., 260 (2001),
  pp.~546--571.

\bibitem{DeLiTs00}
{\sc B.~Desjardins, C.-K. Lin, and T.-C. Tso}, {\em Semiclassical limit of the
  derivative nonlinear {S}chr\"odinger equation}, Math. Models Methods Appl.
  Sci., 10 (2000), pp.~261--285.

\bibitem{DuLaSz16}
{\sc {\'E}.~Dumas, D.~Lannes, and J.~Szeftel}, {\em Variants of the focusing
  {NLS} equation: derivation, justification, and open problems related to
  filamentation}, in Laser filamentation, CRM Ser. Math. Phys., Springer, Cham,
  2016, pp.~19--75.

\bibitem{PGX93}
{\sc P.~G{\'e}rard}, {\em Remarques sur l'analyse semi-classique de
  l'\'equation de {S}chr\"odinger non lin\'eaire}, in S\'eminaire sur les
  \'Equations aux D\'eriv\'ees Partielles, 1992--1993, \'Ecole Polytech.,
  Palaiseau, 1993, pp.~Exp.\ No.\ XIII, 13.
\newblock
  {{\href{http://www.numdam.org/numdam-bin/fitem?id=SEDP_1992-1993____A13_0}{{\mdseries\ttfamily
  www.numdam.org}}}}.

\bibitem{GV01}
{\sc J.~Ginibre and G.~Velo}, {\em Long range scattering and modified wave
  operators for some {H}artree type equations. {III}. {G}evrey spaces and low
  dimensions}, J. Differential Equations, 175 (2001), pp.~415--501.

\bibitem{Grenier98}
{\sc E.~Grenier}, {\em Semiclassical limit of the nonlinear {S}chr\"odinger
  equation in small time}, Proc. Amer. Math. Soc., 126 (1998), pp.~523--530.

\bibitem{HaOz94}
{\sc N.~Hayashi and T.~Ozawa}, {\em Finite energy solutions of nonlinear
  {S}chr\"odinger equations of derivative type}, SIAM J. Math. Anal., 25
  (1994), pp.~1488--1503.

\bibitem{JosserandPomeau}
{\sc C.~Josserand and Y.~Pomeau}, {\em Nonlinear aspects of the theory of
  {B}ose-{E}instein condensates}, Nonlinearity, 14 (2001), pp.~R25--R62.

\bibitem{KhRa82}
{\sc M.~Khanna and R.~Rajaram}, {\em Evolution of nonlinear {A}lfv\'en waves
  propagating along the magnetic fields in a collisionless plasma}, J. Plasma
  Phys., 28 (1982), pp.~459--468.

\bibitem{LiSiSu13}
{\sc X.~Liu, G.~Simpson, and C.~Sulem}, {\em Stability of solitary waves for a
  generalized derivative nonlinear {S}chr\"odinger equation}, J. Nonlinear
  Sci., 23 (2013), pp.~557--583.

\bibitem{GuyCauchy}
{\sc G.~M{\'e}tivier}, {\em Remarks on the well-posedness of the nonlinear
  {C}auchy problem}, in Geometric analysis of PDE and several complex
  variables, vol.~368 of Contemp. Math., Amer. Math. Soc., Providence, RI,
  2005, pp.~337--356.

\bibitem{MOMT76}
{\sc K.~Mio, T.~Ogino, K.~Minamy, and S.~Takeda}, {\em Modified nonlinear
  {S}chr\"odinger equation for {A}lfv\'en waves propagating along the magnetic
  fiel in cold plasma}, J. Phys. Soc. Japan, 41 (1976), pp.~265--273.

\bibitem{Mj76}
{\sc E.~Mj{\o}lhus}, {\em On the modulational instability of hydromagnetic
  waves parallel to the magnetic field}, J. Plasma Phys., 16 (1976),
  pp.~321--334.

\bibitem{MoVi11}
{\sc C.~Mouhot and C.~Villani}, {\em On {L}andau damping}, Acta Math., 207
  (2011), pp.~29--201.

\bibitem{Oz96}
{\sc T.~Ozawa}, {\em On the nonlinear {S}chr\"odinger equations of derivative
  type}, Indiana Univ. Math. J., 45 (1996), pp.~137--163.

\bibitem{PiSt}
{\sc L.~Pitaevskii and S.~Stringari}, {\em Bose-{E}instein condensation},
  vol.~116 of International Series of Monographs on Physics, The Clarendon
  Press Oxford University Press, Oxford, 2003.

\bibitem{Sj82}
{\sc J.~Sj{\"o}strand}, {\em Singularit\'es analytiques microlocales}, in
  Ast\'erisque, 95, vol.~95 of Ast\'erisque, Soc. Math. France, Paris, 1982,
  pp.~1--166.

\bibitem{Sulem}
{\sc C.~Sulem and P.-L. Sulem}, {\em The nonlinear {S}chr\"odinger equation,
  Self-focusing and wave collapse}, Springer-Verlag, New York, 1999.

\bibitem{ThomannAnalytic}
{\sc L.~Thomann}, {\em Instabilities for supercritical {S}chr\"odinger
  equations in analytic manifolds}, J. Differential Equations, 245 (2008),
  pp.~249--280.

\bibitem{WaKoIc79}
{\sc M.~Wadati, K.~Konno, and Y.-K. Ichikawa}, {\em A generalization of inverse
  scattering method}, J. Phys. Soc. Japan, 46 (1979), pp.~1965--1966.

\bibitem{ZS}
{\sc V.~E. Zakharov and A.~B. Shabat}, {\em Exact theory of two-dimensional
  self-focusing and one-dimensional self-modulation of waves in nonlinear
  media}, \v Z. \`Eksper. Teoret. Fiz., 61 (1971), pp.~118--134.

\end{thebibliography}

\end{document}